\documentclass[12pt]{article}
\usepackage{tikz}
\usetikzlibrary{calc,fit,shapes.geometric}
\usepackage{subfigure}
\makeatletter
\def\fnum@figure#1{\figurename\nobreakspace\thefigure
       \hspace{0.6em}}                                     
\makeatother
\begin{document}
\newtheorem{theorem}{Theorem}
\newtheorem{corollary}[theorem]{Corollary}
\newtheorem{definition}[theorem]{Definition}
\newtheorem{conjecture}[theorem]{Conjecture}
\newtheorem{problem}[theorem]{Problem}
\newtheorem{lemma}[theorem]{Lemma}
\newtheorem{proposition}[theorem]{Proposition}
\newtheorem{question}[theorem]{Question}
\newenvironment{proof}{\noindent {\bf Proof.}}
                      {\hfill\rule{2mm}{2mm}\par\medskip}
\newcommand{\remark}{\medskip\par\noindent {\bf Remark.~~}}

\newcommand{\JCTB}{{\it J. Combin. Theory Ser. B.}, }
\newcommand{\JCT}{{\it J. Combin. Theory}, }
\newcommand{\JGT}{{\it J. Graph Theory}, }
\newcommand{\ComHung}{{\it Combinatorica}, }
\newcommand{\DAM}{{\it Discrete Applied Math.},}
\newcommand{\DM}{{\it Discrete Math.}, }
\newcommand{\ARS}{{\it Ars Combin.}, }
\newcommand{\SIAMDM}{{\it SIAM J. Discrete Math.}, }
\newcommand{\SIAMADM}{{\it SIAM J. Algebraic Discrete Methods}, }
\newcommand{\SIAMC}{{\it SIAM J. Comput.}, }
\newcommand{\ConAMS}{{\it Contemp. Math. AMS}, }
\newcommand{\TransAMS}{{\it Trans. Amer. Math. Soc.}, }
\newcommand{\AnDM}{{\it Ann. Discrete Math.}, }
\newcommand{\ConNum}{{\it Congr. Numer.}, }
\newcommand{\CJM}{{\it Canad. J. Math.}, }
\newcommand{\JLMS}{{\it J. London Math. Soc.}, }
\newcommand{\PLMS}{{\it Proc. London Math. Soc.}, }
\newcommand{\PAMS}{{\it Proc. Amer. Math. Soc.}, }
\newcommand{\JCMCC}{{\it J. Combin. Math. Combin. Comput.}, }

\long\def\longdelete#1{} \baselineskip 21 pt  

\title{\bf Path covering number and  {\boldmath $L(2,1)$}-labeling number of graphs
\thanks{Supported in part by National Natural
Science Foundation of China (No.10971248 ) and the Fundamental Research
Funds for the Central Universities.}}
\author{Changhong Lu \ \ and\ \  Qing Zhou \\
   Department of Mathematics,\\
   East China Normal University,\\
   Shanghai 200241, P. R. China}

\maketitle

\begin{abstract}
  A {\it path covering} of a graph $G$ is a set of vertex disjoint
paths of $G$ containing all the vertices of $G$. The {\it path
covering number} of $G$, denoted by $P(G)$, is the minimum number of
paths in a path covering of $G$.  An {\sl $k$-$L(2,1)$-labeling} of
a graph $G$ is a mapping $f$ from $V(G)$ to the set
$\{0,1,\ldots,k\}$ such that $|f(u)-f(v)|\ge 2$ if $d_G(u,v)=1$ and
$|f(u)-f(v)|\ge 1$ if $d_G(u,v)=2$. The {\sl $L(2,1)$-labeling
number $\lambda (G)$} of $G$ is the smallest number $k$ such that
$G$ has a $k$-$L(2,1)$-labeling.   The purpose of this paper is to
study path covering number and  $L(2,1)$-labeling number of graphs.
Our main work extends most of results in [On island sequences of
labelings with a condition at distance two, Discrete Applied Maths
158 (2010), 1-7] and can answer an open problem in [On the structure
of graphs with non-surjective $L(2,1)$-labelings, SIAM J. Discrete
Math. 19 (2005), 208-223].

\bigskip
\noindent {\bf Keywords.} $L(2,1)$-labeling, Path covering number,
Hole index，Algorithm
\end{abstract}

\section{Introduction}                        

A {\it path covering} of a graph $G$ is a set of vertex disjoint
paths of $G$ containing all the vertices of $G$. The {\it path
covering number} of $G$, denoted by $P(G)$, is the minimum number of
paths in a path covering of $G$. The {\it minimum path covering} of
$G$ is a path covering  with size $P(G)$. The path covering problem
is to find a minimum path covering of a graph. The path covering
problem has received some alternative names in the literature, such
as optimal path cover \cite{ar1990, sssr1993,wong1999} and path
partition \cite{yan1994,yc1994}. It is evident that the path
covering problem for general graphs is {\em NP}-complete since
finding a path covering, consisting of a single path, corresponds
directly to the Hamiltonian path problem. Polynomial-time algorithms
to  solve the path covering problem have been known for a few special classes of graphs, including trees
\cite{bcm1974,gh1974,mt1975,sl1979}, block graphs \cite{wong1999,yc1994},
interval graphs \cite{ar1990}, circular-arc graphs \cite{hc2006},
cographs \cite{lop1995}, bipartite permutation graphs
\cite{sssr1993}, cocomparability graphs \cite{ddks1992} and
distance-hereditary graphs \cite{hc2007}. The path covering problem
has many practical applications in different areas, including
 mapping parallel programs to parallel architectures \cite{mw1991},
 code optimization \cite{bg1977} and program testing \cite{nh1979}.

It is known that  for a connected graph $G$, there is a spanning
tree $T$ of  $G$ such that $P(G)= P(T)$ \cite{bcm1974}.  Hence, it
is important to determine the path covering number of trees. There
are polynomial-time algorithms to determine the path cover number of
trees, surprisingly, however,
almost no exact values for the path covering number of special
families of trees are known.

 The problem of vertex labeling with a condition at distance two,
 first studied by Griggs and Yeh \cite{gy1992}, is a variation of the $T$-coloring problem introduced
by Hale \cite{h1980}.  An {\it $L(2,1)$-labeling} of a graph $G$ is
a mapping $f$ from the vertex set $V(G)$ to the set of all
nonnegative integers such that $|f(x)-f(y)|\ge 2$ if $d_G(x,y)=1$
and $|f(x)-f(y)|\ge 1$ if $d_G(x,y)=2$, where $d_G(x,y)$ denotes the
distance between the pair of vertices $x,y$. A {\it
$k$-$L(2,1)$-labeling} is an $L(2,1)$-labeling such that no label is
greater than $k$. The {\it $L(2,1)$-labeling number} $\lambda (G)$
of $G$ is the smallest number $k$ such that $G$ has a
$k$-$L(2,1)$-labeling. A $\lambda (G)$-$L(2,1)$-labeling is referred
to simply as a $\lambda$-labeling. It is shown that the
$L(2,1)$-labeling problem is {\em NP}-complete \cite{gmw1994}. In general, it is
hard to determine $\lambda$ even for special graphs. The reader may
consult \cite{s1994,wgm1995,yeh2006,zhou2006,zhou2007} for known results on $\lambda$. The reader is referred
to \cite{Cal2006} for a survey and \cite{ck1996,
gy1992,vls1998} for background information on this
problem.


The following elegant result, proved by Georges, Mauro and
Whittlesey \cite{gmw1994}, explored the relation between
$L(2,1)$-labeling problem and path covering problem.

\begin{theorem} (cf. Theorem 1.1 in \cite{gmw1994}) \label{thm1}
Suppose that $G$ is a graph of $n$ vertices. Let $G^c$ be the
complement of $G$. \\
(1)\ \ $\lambda(G)\le n-1$ if and only if $P(G^c)=1$; \\
(2)\ \ $\lambda(G)=n+P(G^c)-2$ if and only if $P(G^c)\ge 2$.
\end{theorem}

A $k$-$L(2,1)$-labeling is said to have a {\it hole} $h$ with $1\le
h\le k-1$, if the label $h$ is not used. The minimum number of holes
over all $\lambda$-labelings of a graph $G$ is called the {\it hole
index} of $G$ and is denoted by $\rho (G)$. Several papers
\cite{atw2007,fr2003,fr2006,gm2005,gm2005-2,kst2006,lcz2007,lz2007}
have studied $\rho (G)$ and have investigated its connections with
$\lambda (G)$ and $\Delta (G)$, the maximum degree of $G$.

The following result by Georges and Mauro \cite{gm2005} established
relation between $\rho (G)$ and $P(G^c)$.

\begin{theorem} \cite{gm2005} \label{thm}
Let $G$ be a graph on $n$ vertices and $\lambda (G)\ge n-1$. Then
$\rho (G)=P(G^c)-1$.
\end{theorem}

It is not difficult to know that any two holes are non-consecutive
in a $\lambda$-labeling. An {\it island} of a given
$\lambda$-labeling of $G$ with $\rho (G)$ holes is a maximal set of
consecutive integers used by the labeling. The {\it island sequence}
is the ordered sequence of island cardinalities in nondecreasing
order. Figure \ref{fig1} \cite{attw2010} presents two different $\lambda$-labelings of the
complete bipartite graph $K_{2,3}$ with $\lambda=5$ and $\rho=1$,
and inducing the same island sequence $(2,3)$. Figure \ref{fig2} presents
two different $\lambda$-labelings of the non-connected graph
$K_5\cup K_2$ with $\lambda =8$ and $\rho=2$, and inducing two
different island sequences $(1,1,5)$ and $(1,3,3)$, respectively. In
\cite{gm2005}, Georges and Mauro raised the following  question
to inquire about the existence of a connected graph  possessing two
$\lambda$-labelings with different island sequences.

\begin{question} \cite{gm2005}\label{que1}
Is there a  connected graph admitting at least two
distinct island sequences?
\end{question}

A vertex in a graph is {\it heavy} if it has degree greater than
$2$, otherwise we say this vertex is {\it light}.
A {\it heavy edge} of $G$ is an edge incident to two heavy vertices
of $G$. A {\it leaf} in a
graph is a vertex of degree $1$. A {\it vine} of a graph $G$ is
defined as a maximal path in $G$ such that one endpoint is a leaf
and each vertex in $S$ is a light vertex in $G$.  If $G$ is not a
path, then there is a unique heavy vertex  adjacent to one of the
ends of $S$. We call such vertex the {\it center} of vine $S$.
A {\it
generalized star} is a tree which has exactly one heavy vertex and  all its vines have the same number
of vertices. A graph $G$ is {\it $2$-sparse} if $G$ contains no pair
of adjacent vertices of degree greater than $2$. The above notation is first introduced by Adams et al.
\cite{attw2010}  and they solved the above question by studying complements of
$2$-sparse trees.

 \begin{theorem}(cf. Theorem 2.8 in \cite{attw2010})\label{thm2}
Let $T$ be a $2$-sparse tree. If $T$ is neither a path nor a
generalized star, then its complement $T^c$ is connected and admits
at least two different island sequences.
\end{theorem}

They also determined the path covering number of $2$-sparse trees.

\begin{theorem}(cf. Theorem 2.4 in \cite{attw2010})\label{thm3}
Let $T$ be a $2$-sparse tree with $\ell\ge 2$ leaves. Then
$P(T)=\ell-1$.
\end{theorem}

Furthermore, they  determined the path covering number of more
general connected non-cycle $2$-sparse graphs.

\begin{theorem}(cf. Theorem 3.2 in \cite{attw2010})\label{thm4}
 Let $G$ be a connected $2$-sparse graph with $m\ge 1$
edges, $n$ vertices, and $\ell$ leaves. If $G$ is not a cycle, then
$P(G)=\ell+m-n$.
\end{theorem}

Theorem \ref{thm3} and Theorem \ref{thm4} are important because it
adds to the limited library of known path covering numbers. Also,
when combined with the prior results  in Theorem \ref{thm1} and
Theorem \ref{thm}, it implies that the $L(2,1)$-labeling number and
hole index for complements of non-path $2$-sparse trees and for
complements of certain non-cycle $2$-sparse graphs are  determined.
As pointed in \cite{attw2010}, further study in two
directions is encouraged: \\
 $\bullet$ The existence of other families that admit
multiple island sequence;\\
$\bullet$  These results in Theorem \ref{thm3} and Theorem
\ref{thm4} are extended to include more general trees and graphs.

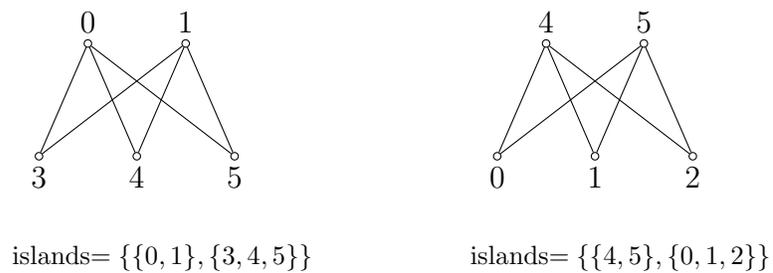
\begin{figure}[!htbp]
\def\thesubfigure{}
\hfill
\subfigure[\mbox{islands$=\{\{0,1\},\{3,4,5\}\}$}]{
\begin{tikzpicture}[
  every node/.style={inner sep=1pt,circle,draw}]
\path \foreach \x/\y in {0/1.95,1/3.25} {
  node[label=above:\x](\x) at(\y,0) {}};
\path[yshift=-1.5cm]
      \foreach \x/\y in {3/1.3,4/2.6,5/3.9} {
  node[label=below:\x](\x) at(\y,0) {}};
\foreach \x in {0,1}
  \foreach \y in {3,4,5}
    \draw (\x) -- (\y);
\end{tikzpicture}}\hfill
\subfigure[\mbox{islands$=\{\{4,5\},\{0,1,2\}\}$}]{
\begin{tikzpicture}[
  every node/.style={inner sep=1pt,circle,draw}]
\path \foreach \x/\y in {4/1.95,5/3.25} {
  node[label=above:\x](\x) at(\y,0) {}};
\path[yshift=-1.5cm]
      \foreach \x/\y in {0/1.3,1/2.6,2/3.9} {
  node[label=below:\x](\x) at(\y,0) {}};
\foreach \x in {4,5}
  \foreach \y in {0,1,2}
    \draw (\x) -- (\y);
\end{tikzpicture}}
\hfill\null
\caption{Two $5$-labelings of $K_{2,3}$ with $1$
hole and island sequence $(2,3)$.}\label{fig1}
\end{figure}

\begin{figure}[!htbp]
\def\thesubfigure{}
\hfill
\subfigure[\mbox{islands$=\{\{0,1,2,3,4\},\{6\},\{8\}\}$}]{
\begin{tikzpicture}[
  every node/.style={inner sep=1pt,circle,draw}]
\path \foreach \x/\y in {0/90,2/18,4/306,6/234,8/162} {
  node[label=\y:\x](\x) at(\y:1.3) {}};
\path[yshift=-1cm] \foreach \x/\y in {3/306,1/234} {
  node[label=\y:\x](\x) at(\y:1.3) {}};
\draw (0) -- (2) -- (4) -- (6) -- (8) -- (0)
      (0) -- (4) -- (8) -- (2) -- (6) -- (0)
      (1) -- (3);
\end{tikzpicture}}\hfill
\subfigure[\mbox{islands$=\{\{0,1,2\},\{4,5,6\},\{8\}\}$}]{
\begin{tikzpicture}[
  every node/.style={inner sep=1pt,circle,draw}]
\path \foreach \x/\y in {0/90,2/18,4/306,6/234,8/162} {
  node[label=\y:\x](\x) at(\y:1.3) {}};
\path[yshift=-1cm] \foreach \x/\y in {5/306,1/234} {
  node[label=\y:\x](\x) at(\y:1.3) {}};
\draw (0) -- (2) -- (4) -- (6) -- (8) -- (0)
      (0) -- (4) -- (8) -- (2) -- (6) -- (0)
      (1) -- (5);
\end{tikzpicture}}
\hfill\null
\caption{Two  $8$-labelings of $K_5\cup K_2$ with  different island sequences.}\label{fig2}
\end{figure}
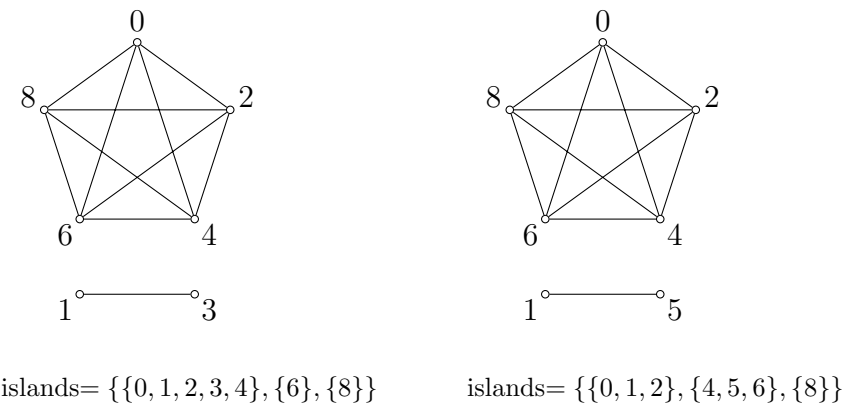

The main purpose of this paper is to extend the work in
\cite{attw2010} through the investigation of the path covering
number  of trees and tree-like graphs. In Section $2$, we give some lemmas which will used in Sections $3$, $4$ and $5$.
In Section $3$, we study the
path covering numbers of trees and get some general results. In
Section $4$, we establish a characterization for trees, whose
complements admit unique island sequences. Furthermore, a
linear-time algorithm  is given to determine  whether the complement of a given tree $T$ admits unique island sequence.
 In Section $5$, we extend some
results on trees to some families of graphs which contain $2$-sparse
graphs as subclass.

\section{Preliminaries}

If $e$ is an edge of a graph $G$, then $G-e$ is the graph obtained
by deleting $e$ from $G$. If $H$ is a subgraph of $G$, the graph
$G-H$ is the graph obtained by deleting the vertices of $H$ from $G$
and any edge incident to a vertex in $H$. If $f$ is an  edge not in
$G$ but its ends are in $G$ then $G+f$ is the graph obtained by
adding $f$ to $G$. If two graphs $G$ and $H$ are disjoint, the graph
$G+H$ is defined as the graph with vertex and edge sets given
respectively by the union of the vertex and edge sets of $G$ and
$H$.

The following  result is obvious.

\begin{lemma} \label{lem1} 
Let $S$ be a vine of a graph $G$. Then $S$ is a subgraph of every
minimum path covering of $G$.
\end{lemma}


\begin{lemma}\label{lem2}
If $v$ is the common center of  vines $S_1$ and $S_2$ in a graph
$G$, then $v$ is an internal vertex in every minimum path covering
of $G$.
\end{lemma}

\begin{proof}
Assume for contradiction that $v$ is the end of path $P$ in a
minimum path covering of $G$. Since $v$ is the center of both $S_1$
and $S_2$, there exists a vine, say $S_2$,  which is not contained
in $P$. By Lemma \ref{lem1}, $S_2$ must be a path in this path
covering, and we call it $Q$. Let $f$ be the edge incident to $v$
and to one of the end of $Q$. Since $P$ and $Q$ are different,
$(P+Q)+f$ is a path. Replacing paths $P$ and $Q$ with $(P+Q)+f$, we
obtain a path covering of $G$ with smaller number of paths, a
contradiction.
\end{proof}

We now illustrate the {\it swapping construction} that is first
introduced in \cite{attw2010}. Let $P$ and $Q$ be two different
paths in a given path covering of a graph $G$ such that $P$ contains
an edge $e$ incident to an internal vertex $v$ of $P$, and one end
of $Q$ is adjacent to $v$ through an edge $f$. Clearly, $P-e$ has
two connected components, namely the paths $P_1$ and $P_2$ where $v$
is an end of $P_1$. Since $P$ and $Q$ are different paths,
$(P_1+Q)+f$ is a path. We can replace paths $P$ and $Q$ with paths
$(P_1+Q)+f$ and $P_2$ to obtain another path covering of $G$ with
the same number of paths. For convenience, we will say that this new
path covering was obtained from the original one by {\it swapping
$e$ with $f$}. 


\begin{lemma}\label{lem3}
Let $v$ be a heavy vertex of a tree $T$ and let $e$ be an edge
incident to $v$. Then $e$ is not used in some minimum path covering
of $T$ and hence $P(T)=P(T-e)$ if one of the following conditions
holds:\\
(a)\ \ $v$ is the common center of at least three vines;\\
(b)\ \ $v$ is the common center of exactly two vines and $e$ is not
incident to any end of vines.
\end{lemma}

\begin{proof}
Suppose that one of the above condition holds. Consider an arbitrary
minimum path covering of $T$, by Lemma \ref{lem2}, $v$ is an
internal vertex in this minimum path covering. Assume that  a path $P$ in this
path covering contains $v$. Then $P$ contains  exactly two edges
incident to $v$. If $e\notin E(P)$, then $e$ is  contained in no
paths within the given minimum path covering  since $v$ is an
internal vertex in this minimum path covering. It is clear that
$P(T)=P(T-e)$. Hence, we assume that $e\in E(P)$. In this situation,
there exists a vine, say $S$, which is not contained in the path
$P$, and by Lemma \ref{lem1}, $S$ must be a path in this path
covering, and we call it $Q$. Let $f$ be the edge incident to $v$
and to one of the end of $Q$. By swapping $e$ with $f$, we obtain
another minimum path covering of $T$ and $e$ is not contained in any
path within the new path covering. It implies that $P(T)=P(T-e)$.
\end{proof}

\begin{lemma}\label{lem4}
If $G$ is a  graph such that each heavy vertex  has at
least three light neighbors, then $P(G)=P(G-e)$, where $e$ is a heavy
edge of $G$.
\end{lemma}

\begin{proof}
Let $e=uv$ be a heavy edge of $G$. Clearly, $P(G-e)\ge P(G)$.
Consider an arbitrary minimum path covering of $G$, we will use this
minimum path covering to construct a path covering of $G-e$ with
exactly $P(G)$ paths, which would imply that $P(G-e)=P(G)$. Suppose
that the path $P$ contains the vertex $u$ in this minimum path
covering. If $e$ is not in $P$, then this path covering of $G$ is
obviously a path covering of $G-e$ with $P(G)$ paths. So, we will
consider the case when $e$ is in $P$. Let $u_1,\ldots,u_k$ ($k\ge
3$) be light vertices adjacent to $u$ and let $v_1,\ldots,v_t$
($t\ge 3$) be light vertices adjacent to $v$. $P$ contains at most
four vertices in $\{u_1,\ldots,u_k,v_1,\ldots,v_t\}$. 
Hence, there is
a vertex, say $u_3$, not containing in $P$. Since $u_3$ is a light
vertex, $u_3$ must be an end of one path, call it $Q$, in this
minimum path covering of $G$. Let $P_1$ be the connected component
of $P-e$ containing $u$ and let $P_2$ be the other component. Let
$f$ be the edge incident to $u$ and $u_3$. Then $P'$= $(P_1+Q)+f$ is
a path. Replacing $P$ and $Q$ with $P'$ and $P_2$ we obtain a path
covering of $G-e$ with $P(G)$ paths.
\end{proof}


\section{Path covering number of trees}                            

In this section, we establish some general results for the path
covering number of trees.

\begin{theorem}\label{thm5}
Let $T$ be a tree with $\ell$ leaves and $h$ heavy edges. If $T$ is
not a single vertex, then
$$ \ell-h-1\le P(T)\le \ell-1.$$
\end{theorem}

\begin{proof}
The proof will proceed by induction on the number of vertices in $T$.
The result clearly holds  when $T$  is a star or $T$ has exactly two vertices. Suppose now that  $T$ is other than a star and has at
least three vertices.
We can choose a vertex $v\in V(T)$ with exactly one non-leaf neighbor $u$ and $k$ leaf neighbors $z_1,\cdots,z_k$.
It is  known and easy to see that (PT1) and (PT2) hold.

(PT1) If $k=1$, then $P(T)=P(T')$, where $T'=T-z_1$;

(PT2) If $k\ge 2$, then $P(T)=P(T')+k-1$, where $T'=T-\{v,z_1,\cdots,z_k\}$.

If  $k=1$, then $v$ becomes a leaf in $T'$ and hence $T'$ has $\ell'=\ell$ leaves. So we have $P(T)=P(T')$ (By (PT1)), $\ell=\ell'$ and $h=h'$.
 By the induction hypothesis, we have $\ell'-h'-1\le P(T')\le \ell'-1$. Hence, $\ell-h-1\le P(T)\le \ell-1$.
  If  $k\ge 2$ and $d_T(u)=2$, then $u$ becomes a leaf in $T'$ and hence $T'$ has $\ell'=\ell-k+1$ leaves, we have  $P(T)=P(T')+k-1$ (By (PT2)), $\ell=\ell'+k-1$ and $h=h'$, we can again apply the induction. If  $k\ge 2$ and $d_T(u)\ge 3$, we have $P(T)=P(T')+k-1$ (By (PT2)),
  $\ell=\ell'+k$ and $h\ge h'+1$.
Applying the  induction, we  have $ \ell'-h'-1\le P(T')\le \ell'-1.$ and hence $\ell-h-1\le P(T)\le \ell-1$.
\end{proof}

For a tree $T$ with at least two vertices, $T$ is $2$-sparse if and
only if $h=0$. Hence, Theorem \ref{thm3} (cf. Theorem 2.4 in
\cite{attw2010}) is a direct corollary of Theorem \ref{thm5}. In
fact we can get a stronger result for $2$-sparse tree.

\begin{theorem}\label{thm6}
 Let $T$ be a  tree with $\ell\ge 2$ leaves.  Then $T$ is $2$-sparse if and only if $P(T)=\ell-1$.
\end{theorem}

\begin{proof}
{\it Sufficiency.} The result follows from  Theorem \ref{thm5}.

{\it Necessity.} Suppose to the contrary that the number of heavy
edges $h\not= 0$. Let $e$ be a heavy edge of $T$.   $T'$ and $T''$
are two components of $T-e$. It is clear that both $T'$ and $T''$
have at least three vertices since $e$ is a heavy edge of $T$.
Assume that $T'$ ($T''$, respectively) has $\ell'$ ($\ell''$,
respectively) leaves.  By Theorem \ref{thm5},
$P(T')\le \ell'-1$ and  $P(T'')\le \ell''-1.$
Hence,
$$ \ell-1=P(T)\le P(T')+P(T'')\le (\ell'+\ell'')-2.$$

But $\ell'+\ell''=l$ as $e$ is a heavy edge of $T$. Therefore,
$\ell-1=P(T)\le \ell-2.$
This is a contradiction. So, $h=0$ and $T$ is $2$-sparse.
\end{proof}

Now we extend  partial results in Theorem \ref{thm6}  to include
more general trees.

\begin{theorem}\label{thm7}
Let $T$ be a tree with $\ell\ge 2$ leaves in which each heavy
vertex  has at least one light neighbors. Let $S$ be
the subset of $V(T)$ such that each element of $S$ is a heavy vertex
with exactly one light neighbor. Then
 $$ P(T)=\ell-h+s-t-1,$$
where $h$ is the number of heavy edges and $s$ is the size of $S$ and $t$ is the size of a maximum matching of the subgraph induced by $S$.
\end{theorem}

\begin{proof}
 The proof will proceed by induction on the number of vertices in $T$ and use the same method in the proof of Theorem \ref{thm5}.
  When $T$  is a star or $T$ has exactly two vertices, the result clearly holds.   Suppose now that  $T$ is a tree other than a star
   and has at least three vertices. We can also
  choose a vertex $v\in V(T)$ with exactly one non-leaf neighbor $u$ and $k$ leaf neighbors $z_1,\cdots,z_k$.
  Let $S'$ be the subset of $T'$ such that each element of
 $S'$ is a heavy vertex with exactly one light
  neighbor in $T'$ and $t'$ be  the size of a maximum matching of the subgraph induced by $S'$.
  By induction, $P(T')=\ell'-h'+s'-t'-1$, where $\ell'$ ($h'$,
 respectively) is the number of leaves (heavy edges, respectively)
 and $s'$ is the size of  $S'$ in $T'$.

 If $k=1$, then $v$ becomes a leaf in $T'$ and hence $T'$ has $\ell'=\ell$ leaves. So we have $P(T)=P(T')$ (By (PT1)), $\ell=\ell'$, $h=h'$,
 $s=s'$ and $t'=t$. Applying induction, we clearly have $P(T)=\ell-h+s-t-1$. Now, we assume that $k\ge 2$.

 If $d_T(u)=2$,  then $\ell'=\ell-k+1$, $h'=h$, $s'=s$ and $t=t'$.
 Hence, $P(T)=P(T')+k-1=\ell-h+s-t-1$ as desired. If  $d_T(u)=3$   with $u\notin S$ or $d_T(u)\ge 4$,
 then $\ell'=\ell-k$, $h'=h-1$, $s'=s$ and $t'=t$. Hence,  $P(T)=P(T')+k-1=\ell-h+s-t-1$ as desired.
  Now we assume that  $u$ has three neighbors $v$, $u_1$ and $u_2$ and $u\in S$.  Without loss of generality,
 we  assume that $u_1$ is a light vertex and $u_2$ is  a heavy vertex in $T$. If $u_2\notin S$, then
  $\ell'=\ell-k$, $h'=h-2$, $s'=s-1$ and $t'=t$. Hence, $P(T)=P(T')+k-1=\ell-h+s-t-1$ as
  desired. Now we consider the case $u_2\in S$.  In this situation, $d_{T'}(u)=2$,
  $uu_2$ is not a heavy edge in $T'$ and hence $u_2\notin S'$. Thus  $\ell'=\ell-k$, $h'=h-2$ and $s'=s-2$.
  Let $C$ denote the connected component containing $u$ and $u_2$ in the subgraph  induced by $S$.
    Note that $C$ is a tree in which $u$ is a leaf. Let $C'=C-uu_2$. So, it is easy to
   know that the size of a maximum matching in $C'$ than that in $S$ is small one.
 So $t'=t-1$ and Hence  $P(T)=P(T')+k-1=\ell-h+s-t-1$ as
  desired.
\end{proof}

  We say a graph is {\it general $2$-sparse} if each its heavy vertex
has at least two light neighbors. Obviously, a $2$-sparse
graph must be  general $2$-sparse. Furthermore, we have

\begin{theorem}\label{thm8}
 Let $T$ be a  tree with $\ell\ge 2$ leaves and $h$  heavy edges.
 Then $T$ is general $2$-sparse if and only if  $P(T)=\ell-h-1$.
\end{theorem}

\begin{proof}
{\it  Sufficiency.} It follows from Theorem \ref{thm7} by using
$s=0$ and $t=0$.

{\it Necessity.} We use induction on $h$. If $h=0$, then
$P(T)=\ell-1$ and hence it is ok by Theorem \ref{thm6}. We assume
that $h\not=0$. Let $r$ be an arbitrary vertex of $T$ and let $e=uv$
be a heavy edge of $T$ such that the distance between $r$ and $e$ is
largest. Furthermore, we assume that  $d_T(r,u)=d_T(r,v)-1$.  By the
selection of $e$,  all neighbors of $v$ except that $u$, say
$v_1,\ldots, v_k$ ($k\ge 2$), are ends of vines.
 By Lemma \ref{lem3}, $e$ is not used in some minimum path
covering of $T$, and hence  $P(T)=P(T-e)$. Suppose that $T'$ ($T''$,
respectively) is the connected component of $T-e$ containing $u$
($v$, respectively).  It is clear that both $T'$ and $T''$ have at
least three vertices. Assume that $T'$ ($T''$, respectively) has
$\ell'$ ($\ell''$, respectively) leaves and $h'$ ($h''$,
respectively) heavy edges. By Theorem \ref{thm5},
$P(T')\ge \ell'-h'-1$ and $P(T'')\ge \ell''-h''-1.$ Note that
$\ell'+\ell''=l$ and  $ h'+h''+1\le h.$ Hence, $$\ell-h-1=P(T)=P(T')+P(T'')\ge \ell'-h'-1+\ell''-h''-1\ge \ell-h-1.$$


This implies that $P(T')=\ell'-h'-1$, $P(T'')=\ell''-h''-1$ and
$h'+h''+1=h$ hold together. By induction hypothesis, we have both
$T'$ and $T''$ are general $2$-sparse, and $T$ is obtained from $T'$
and $T''$ by adding an edge $e=uv$. It is obvious that the vertex $v$ has least
two light neighbors. If $d_{T'}(u)\ge 3$, it is clear
that any heavy vertex in $T'$ has  at least two
light neighbors when $e$ is added. We assume that $d_{T'}(u)= 2$ and
$u_1, u_2$ are its neighbors in $T'$. We claim that $u_1$ and $u_2$
are both light vertices. If not, without loss of generality, we
assume that $u_1$ is a heavy vertex in $T$. Let $f=uu_1$. Then $f$
is a heavy edge in $T$, but it is not a heavy edge in $T'$ as
$d_{T'}(u)=2$. So, $h'+h''\le h-2$, it contracts with $h'+h''+1=h$.
Therefore, $u_1$ and $u_2$ are both light vertices in $T$, and hence
$T$ is general $2$-sparse.
\end{proof}

 The following result is a direct corollary of Theorem
 \ref{thm1}, Theorem \ref{thm} and Theorem \ref{thm7}.

 \begin{corollary}\label{cor1}
 Let $T$ be a non-path tree satisfying the conditions in Theorem \ref{thm7}.
 Then $\lambda (T^c)=n+\ell-h+s-t-3$ and $\rho
 (T^c)=\ell-h+s-t-2$,  where $n$ ($\ell$, $h$, respectively) is the number of vertices (leaves, heavy
edges, respectively) and $s$ ($t$, respectively) is the size of $S$ (a maximum matching of the subgraph induced by $S$, respectively).
 \end{corollary}

At the end of this section, we remark that it would be interesting
to investigate the path covering number of other more general
families of trees.

\section{Island sequences for  complements of trees}                                  

Given a minimum path covering $\cal P$ of a graph $G$, the {\it path
sequence} of  $\cal P$ is the ordered sequence of the  numbers of
vertices of paths in $\cal P$ in nondecreasing order (note that this
definition allows for repeated cardinalities).  As shown in
\cite{gmw1994}, if $P(G)\ge 2$, then a minimum path covering $\cal
P$ of a graph $G$ can induce a $\lambda (G^c)$-labeling $f_{\cal P}$
of $G^c$ with $P(G)-1$ holes and  the  island sequence of $f_{\cal
P}$ is same as the path sequence of  $\cal P$. We refer readers to
\cite{gmw1994} for the complete proof. Hence, $G^c$ admits multiple
 island sequences if and only if $G$ admits distinct path
sequences and $P(G)\ge 2$. In this section, we will establish a
constructive characterization for trees with unique path sequence.

A {\it labeled generalized star} is a generalized star in which all
neighbors of its center are labeled $B$ and other non-leaf vertices
 are labeled $A$. Figure \ref{fig3} (a) illustrates a labeled generalized star
with three vines of length $3$.  For convenience, a path $P$ of
length $2k$ is called a labeled generalized star with two vines of
length $k-1$ in which  two neighbors of its  center are labeled $B$
and other non-leaf vertices  are labeled $A$. Figure \ref{fig3} (b)
illustrates a labeled generalized star with two vines of length $3$. A
{\it labeled path} is a path with at least three vertices in which
all non-leaf vertices are labeled $A$. Figure \ref{fig3} (c) illustrates a
labeled path of length $5$.

\begin{figure}[!htbp]
\hfill
\subfigure[]{
\begin{tikzpicture}[
  every node/.style={inner sep=1pt,circle,draw}]
\path node[label=right:$A$](0) at(0:0) {}
      node[label=right:$B$](1) at(90:.7) {}
      node[label=right:$A$](4) at(90:1.4) {}
      node[label=right:$A$](7) at(90:2.1) {}
      node(10) at(90:2.8) {}
      node[label=30:$B$](2) at(300:.7) {}
      node[label=30:$A$](5) at(300:1.4) {}
      node[label=30:$A$](8) at(300:2.1) {}
      node(11) at(300:2.8) {}
      node[label=150:$B$](3) at(240:.7) {}
      node[label=150:$A$](6) at(240:1.4) {}
      node[label=150:$A$](9) at(240:2.1) {}
      node(12) at(240:2.8) {};
\draw (10) -- (7) -- (4) -- (1) -- (0) -- (2) -- (5) -- (8) -- (11)
      (0) -- (3) -- (6) -- (9) -- (12);
\useasboundingbox(-1.4,-2.8);
\end{tikzpicture}}
\hfill
\subfigure[]{
\begin{tikzpicture}[
  every node/.style={inner sep=1pt,circle,draw}]
\path node[label=right:$A$](0) at(0:0) {}
      node[label=right:$B$](1) at(90:.7) {}
      node[label=right:$A$](3) at(90:1.4) {}
      node[label=right:$A$](5) at(90:2.1) {}
      node(7) at(90:2.8) {}
      node[label=right:$B$](2) at(270:.7) {}
      node[label=right:$A$](4) at(270:1.4) {}
      node[label=right:$A$](6) at(270:2.1) {}
      node(8) at(270:2.8) {};
\draw (7) -- (5) -- (3) -- (1) -- (0) -- (2) -- (4) -- (6) -- (8);
\end{tikzpicture}}
\hfill
\subfigure[]{
\begin{tikzpicture}[
  every node/.style={inner sep=1pt,circle,draw}]
\path node[label=right:$A$](0) at(0:0) {}
      node[label=right:$A$](1) at(90:.7) {}
      node[label=right:$A$](3) at(90:1.4) {}
      node(5) at(90:2.1) {}
      node[label=right:$A$](2) at(270:.7) {}
      node(4) at(270:1.4) {};
\draw (5) -- (3) -- (1) -- (0) -- (2) -- (4);
\useasboundingbox(0,-2.45);
\end{tikzpicture}}
\hfill\null
\caption{}\label{fig3}
\end{figure}

To state the constructive characterization of trees with unique path
sequence, we need to introduce  a family of labeled trees and three
types of operations.

A  family of labeled  trees $\mathcal {F}$ is defined as
$\mathcal{F}$ $=\{T~|~T$ is obtained from a labeled generalized star
with at least three vines or a labeled path by a finite sequence of
 operations of Type-1, Type-2 or Type-3$\}$.

\vskip 0.2cm
 Let $T\in \mathcal {F}$ be a labeled tree in which all non-leaf vertices are
labeled $A$ or $B$.

\noindent{\bf Type-1 operation:}   Attach a labeled generalized star $S$
with at least three vines  to $T$ by adding an edge $uv$, where $u$
is a vertex in $T$ labeled $A$ and $v$ is the center of $S$. Figure \ref{fig4} (a) illustrates this operation.\\

\noindent{\bf Type-2 operation:}   Attach a labeled path $P$ to $T$
by adding an edge $uv$, where $u$ is a vertex in $T$ labeled $A$ and
$v$ is a non-leaf vertex in $P$.  Figure \ref{fig4} (b) illustrates this operation.\\

\noindent{\bf Type-3 operation:}   Attach a labeled generalized star $S$
to $T$ by adding an edge $uv$, where  $v$ is the center of $S$ and
$u$ is a vertex in $T$ labeled $B$ such that  the length of each
vine in $S$ is same as that of  the labeled generalized star $S'$, which
is  the original/attached  labeled generalized star  containing $u$.  Figure \ref{fig4} (c) illustrates this operation.\\


\begin{figure}[!htbp]
\hfill
\subfigure[Type-1 operation]{
\begin{tikzpicture}[
  vertex/.style={inner sep=1pt,circle,draw}]
\path node[vertex,label=30:$A$](0) at(0:0) {}
      node[vertex,label=right:$A$](1) at(90:.7) {}
      node[vertex](4) at(90:1.4) {}
      node[vertex,label=60:$A$,label=left:$u$](2) at(330:.7) {}
      node[vertex](5) at(330:1.4) {};
\path[xshift=6mm,yshift=-2.1cm]
      node[vertex,label=right:$v$,label=left:$A$](10) at(0:0) {}
      node[vertex,label=below right:$B$](11) at(270:.7) {}
      node[vertex,label=below right:$A$](14) at(270:1.4) {}
      node[vertex](17) at(270:2.1) {}
      node[vertex,label=30:$B$](12) at(300:.7) {}
      node[vertex,label=30:$A$](15) at(300:1.4) {}
      node[vertex](18) at(300:2.1) {}
      node[vertex,label=150:$B$](13) at(240:.7) {}
      node[vertex,label=150:$A$](16) at(240:1.4) {}
      node[vertex](19) at(240:2.1) {};
\draw (4) -- (1) -- (0) -- (2) -- (5);
\draw (17) -- (14) -- (11) -- (10) -- (12) -- (15) -- (18)
      (10) -- (13) -- (16) -- (19);
\draw (2) -- (10);
\node[draw,dashed,ellipse,fit=(4)(5),label=right:$T$] {};
\node[draw,dashed,ellipse,fit=(10)(18)(19),label=right:$S$] {};
\end{tikzpicture}}
\hfill
\subfigure[Type-2 operation]{
\begin{tikzpicture}[
  vertex/.style={inner sep=1pt,circle,draw}]
\path node[vertex,label=30:$A$](0) at(0:0) {}
      node[vertex,label=right:$A$](1) at(90:.7) {}
      node[vertex](4) at(90:1.4) {}
      node[vertex,label=60:$A$,label=left:$u$](2) at(330:.7) {}
      node[vertex](5) at(330:1.4) {};
\path[xshift=6mm,yshift=-2.1cm]
      node[vertex,label=right:$v$,label=left:$A$](10) at(0:0) {}
      node[vertex](12) at(300:.7) {}
      node[vertex,label=150:$A$](13) at(240:.7) {}
      node[vertex,label=150:$A$](16) at(240:1.4) {}
      node[vertex](19) at(240:2.1) {};
\draw (4) -- (1) -- (0) -- (2) -- (5);
\draw (10) -- (12)
      (10) -- (13) -- (16) -- (19);
\draw (2) -- (10);
\node[draw,dashed,ellipse,fit=(4)(5),label=right:$T$] {};
\node[draw,dashed,ellipse,fit=(10)(12)(19),label=right:$P$] {};
\useasboundingbox(0,-4.2);
\end{tikzpicture}}
\hfill
\subfigure[Type-3 operation]{
\begin{tikzpicture}[
  vertex/.style={inner sep=1pt,circle,draw}]
\path node[vertex,label=30:$A$](0) at(0:0) {}
      node[vertex,label=right:$B$](1) at(90:.7) {}
      node[vertex,label=right:$A$](4) at(90:1.4) {}
      node[vertex](7) at(90:2.1) {}
      node[vertex,label=60:$B$,label=left:$u$](2) at(330:.7) {}
      node[vertex,label=60:$A$](5) at(330:1.4) {}
      node[vertex](8) at(330:2.1) {}
      node[vertex,label=120:$B$](3) at(210:.7) {}
      node[vertex,label=120:$A$](6) at(210:1.4) {}
      node[vertex](9) at(210:2.1) {};
\path[xshift=6mm,yshift=-2.1cm]
      node[vertex,label=right:$v$,label=left:$A$](10) at(0:0) {}
      node[vertex,label=below right:$B$](11) at(270:.7) {}
      node[vertex,label=below right:$A$](14) at(270:1.4) {}
      node[vertex](17) at(270:2.1) {}
      node[vertex,label=30:$B$](12) at(300:.7) {}
      node[vertex,label=30:$A$](15) at(300:1.4) {}
      node[vertex](18) at(300:2.1) {}
      node[vertex,label=150:$B$](13) at(240:.7) {}
      node[vertex,label=150:$A$](16) at(240:1.4) {}
      node[vertex](19) at(240:2.1) {};
\draw (7) -- (4) -- (1) -- (0) -- (2) -- (5) -- (8)
      (0) -- (3) -- (6) -- (9);
\draw (17) -- (14) -- (11) -- (10) -- (12) -- (15) -- (18)
      (10) -- (13) -- (16) -- (19);
\draw (2) -- (10);
\draw[dashed] (7)+(0:4mm) arc(0:180:4mm)
     .. controls(150:10mm) .. ($(9)+(130:4mm)$) arc(130:310:4mm)
     .. controls(270:10mm) .. ($(8)+(240:4mm)$) arc(-120:60:4mm)
     .. controls(30:10mm) .. node[label=30:$T$] {} ($(7)+(0:4mm)$);
\node[draw,dashed,ellipse,fit=(10)(18)(19),label=right:$S$] {};
\end{tikzpicture}}
\hfill\null
\caption{}\label{fig4}
\end{figure}

\begin{theorem}\label{thm9}
Every labeled tree in  $\mathcal F$ admits unique path sequence.
\end{theorem}

\begin{proof}
Suppose $T$ is  a labeled tree in $\mathcal F$. Let $s(T)$ denote
the number of operations required to construct $T$. We need prove a
more general statement by induction on $s(T)$ as follows:

\noindent(C1) $T$ admits unique path sequence;

\noindent(C2) Every vertex  labeled $A$ is an internal vertex in
any minimum path covering of $T$;

\noindent(C3) For each vertex labeled $B$, if it is one end of a
path in some minimum path covering of $T$, then the path is exactly
the vine, containing this vertex, of the  labeled generalized star
in the construction process of $T$.

If $s(T)=0$, then $T$ is a labeled generalized star with at least three
vines or a labeled path. It is obvious that $T$ satisfies (C1)-(C3),
and hence the assertion is true.

Assume that $T'$ satisfies  (C1)-(C3) for all trees $T'\in \mathcal
{F}$ with $s(T')<k$, where $k\ge 1$ is an integer. Let $T$ be a tree
with $s(T)=k$. Then $T$ is obtained from a tree $T'$ by one of a
Type-1, a Type-2 or a Type-3 operation, and $T'\in \mathcal {F}$.
Applying the inductive hypothesis to $T'$, $T'$ satisfies (C1)-(C3).

If $T$ is obtained from a tree $T'$ by a Type-1 operation, by Lemma
\ref{lem3}, $P(T)=P(T')+\ell-1$, where $\ell\ge 3$ is the number of leaves
in the generalized star $S$. Suppose that  vertices labeled $B$ in $S$
are $v_1,\ldots,v_{\ell}$. Let $e=uv$ and $e_i=vv_i$ for $i=1,\ldots,\ell$.
Given a minimum path covering of $T$, by Lemma \ref{lem2}, $v$ is an
internal vertex in a path $P$. If $P$ contains $e$, by swapping $e$
with some $e_i$, we obtain another minimum path covering of $T$,
which also induces a minimum path covering of $T'$ and $e$ is not
used. Meanwhile, $u$ is one of ends of a path in this minimum path
covering of $T'$. It contradicts with the fact $T'$ satisfies (C2),
i.e., every vertex in $T'$ labeled $A$ is an internal vertex in
every minimum path covering of $T'$. Hence, $e=uv$ is not used in
any minimum path covering of $T$.  Since both $T'$ and $S$ satisfy
(C1), $T$ admits unique path sequence, i.e.,  $T$ satisfies (C1).
By induction hypothesis, every vertex in $T'$ clearly satisfies
(C2)-(C3). Lemma \ref{lem1} implies that every vertex labeled $A$ in
$S$ satisfies (C2). Let $P$ be the path containing the vertex $v$ in
a minimum path covering of $S$. Without loss of generality, we
assume that $e_1\in E(P)$ and $e_2\in E(P)$. Then the vine
containing $v_i$ in $S$ ($3\le i\le \ell$) is exactly a path in this
minimum path covering of $S$.  By swapping $e_1$ ($e_2$,
respectively) with $e_{\ell}$, we can get  a new minimum path covering of
$S$ such that the vine containing $v_1$ ($v_2$, respectively) in $S$
is exactly a path in this minimum path covering. Therefore, every
vertex in $S$ satisfies (C3).

If $T$ is obtained from a tree $T'$ by a Type-2 operation, we can
use similar arguments as above and get that $e=uv$ is not used in
any minimum path covering of $T$, and hence the labeled path $P$
attached into $T'$ is exactly a path in any minimum path covering of
$T$. It is easy to check that $T$ satisfies  (C1)-(C3) in this
situation.

If $T$ is obtained from a tree $T'$ by a Type-3 operation,  Lemma
\ref{lem3} implies that  $P(T)=P(T')+\ell-1$, where $\ell$ is the number
of leaves in the labeled generalized star $S$. By induction hypothesis,
both $T'$ and $S$ have unique path sequence. Suppose that the unique
path sequence of $T'$ is $(x_1,\ldots,x_{P(T')})$. Obviously,  the
unique path sequence of $S$ is $( 2s_1+1, s_1,\ldots,s_1)$, where
the length of  path sequence of $S$ is $\ell-1$ and $s_1$ is the number
of vertices in each vine of $S$. It is clear that
$(x_1,\ldots,x_{P(T')}, s_1,\ldots,s_1, 2s_1+1)$ is a path sequence
of $T$ if we do not consider nondecreasing order. When the edge $uv$
is not used in a minimum path covering of $T$, its path sequence is
clearly $(x_1,\ldots,x_{P(T')}, s_1,\ldots,s_1, 2s_1+1)$ if the
nondecreasing order is not considered. Hence we assume that $uv$ is
used in a minimum path covering $\mathcal P$ of $T$ and $P\in
\mathcal P$  is the path  containing the edge $uv$ and  with $p$
vertices. By Lemmas \ref{lem1} and \ref{lem2}, there is a vine of $S$
as a path in $\mathcal P$. By swapping  the edge $uv$ and the edge
$wv$, where $w$ is one of ends of this vine, we get a new minimum
path covering $\mathcal P'$, in which the edge $uv$ is not used.  In
$\mathcal P'$, $v$ is an end of a path, say $P'$, with $p-(s_1+1)$
vertices. As $\mathcal P'$ restricted to the tree $T'$ induce  a
minimum path covering of $T''$,  by induction, $P'$ is exactly the
vine containing $v$ in the  labeled generalized star. Type-3
operation implies that $p-(s_1+1)=s_1$. Hence, $p=2s_1+1$. Since
both $T'$ and $S$ admit the unique path sequence, the  path sequence
induced by $\mathcal P'$ is $(x_1,\ldots,x_{P(T')}, s_1,\ldots,s_1,
2s_1+1)$ if the nondecreasing order is not considered. Note that
$\mathcal P$ can also be obtained from $\mathcal P'$ by swapping
$wv$ and $uv$. Hence,  $\mathcal P$ admits the same path sequence
with $\mathcal P'$. Therefore, $T$ satisfies (C1).

Let $w$ be a vertex in $T$ labeled $A$. If $w$ is a vertex in $S$,
Lemmas \ref{lem1} and \ref{lem2} implies that $w$ is an internal
vertex in every minimum path covering of $T$. We assume that $w\in
V(T')$. Let $\mathcal P$ be a minimum path covering of $T$ in which
$w$ is an end of a path.  Without loss of generality, we assume that
the edge $uv$ is used in $\mathcal P$, and hence there is a path
$P\in \mathcal P$ containing $uv$. Similarly, by swapping $uv$ and
an edge incident to $v$, we can get a new minimum path covering
$\mathcal P'$ of $T$  in which $uv$ is not used. Hence, $\mathcal
P'$ restricted to $T'$ induces a minimum path covering of $T'$.
Since $w$ is an end of a path in $\mathcal P$, then $w$ is also an
end of a path in this minimum path covering of $T'$. However, by
induction hypothesis, $T'$ satisfies (C2). It is a contradiction.
Therefore, $w$ must be an internal vertex in $\mathcal P$, and hence
$T$ satisfies (C2).

 By Lemma \ref{lem3},  any minimum path covering of $T'$ and any minimum path covering of $S$
consists of a minimum path covering of $T$. By induction hypothesis,
both $T'$ and $S$ satisfy (C3). Hence, it is easy to see that $T$
satisfies (C3).
\end{proof}

\begin{theorem}\label{thm10}
Let $T$ be a tree with at least three vertices. If $T$ admits unique
path sequence, then  we can labeled $T$  by $A$ and $B$ such that it
is an element of $\mathcal F$.
\end{theorem}

\begin{proof}
Suppose $T$ is  a tree with at least three vertices and  it admits
unique path sequence.  We use induction on $hv(T)$, the number of
heavy vertices in $T$, to prove a more general statement as follows:

\noindent(D1) $T$ can be labeled by $A$ and $B$ such that it is an
element of $\mathcal F$;

\noindent(D2) each vertex  labeled $A$ is an internal vertex in
every minimum path covering of $T$;

\noindent(D3) each vertex labeled $B$ is an end of the path, induced
by the vine containing this vertex in the  labeled generalized
star, in some minimum path covering of $T$.

The assertion is clearly true if $hv(T)=0,1$. We now assume that
$hv(T)\ge 2$ and the assertion holds for smaller values of $hv(T)$.

Let $r$ be an arbitrary vertex in $T$ and let $v$ be the heavy
vertex such that  $d_T(r,v)$ is largest. Then there is a unique
vertex, say $u$, with $d_T(r,u)=d_T(r,v)-1$. Let $e$ be the edge
incident to $u$ and $v$. By the selection of $v$, $v$ has at least
two neighbors different from $u$, each of which is one end of a
vine. By Lemma \ref{lem3}, $e$ is not used  in some minimum path
covering of $T$ and hence $P(T)=P(T-e)=P(T')+P(T'')=P(T')+d_T(v)-2$,
where $T'$($T''$, respectively) is the component containing $u$
($v$, respectively) of $T-e$. It is easy to see that $T'$ admits
unique path sequence since otherwise $T$ admits different path
sequences. Hence, by induction hypothesis,  (D1)-(D3) hold for the
tree $T'$. Suppose that neighbors of $v$ are $u,v_1,\ldots, v_k$,
where $k=d_T(v)-1\ge 2$. Let $S_i$ be the vine containing $v_i$  and
$e_i=vv_i$ for $i\in \{1, 2,\cdots,k\}$.

 \vskip 0.2cm

{\it Case 1. $u$ is labeled $A$ in the label tree $T'$.}

It implies that $u$ is an internal vertex in every minimum path
covering of $T'$.  If there is a path $P$ containing $e$ in a
minimum path covering of $T$, then by swapping $e$ and $e_i$, where $e_i$ is not contained in $P$, we get a new
minimum path covering of $T$, in which $e$ is not used.
Restricted to $T'$, this new minimum path covering induces a minimum
path covering of $T'$ such that $u$ is an end of a path. It
contradicts that $u$ is an internal vertex in every minimum path
covering of $T'$. Therefore, $e$ is not used in any minimum path
covering of $T$.

If $k=2$, then $T''$ is a path with at least three vertices, and
hence we label each non-leaf vertex of this path by $A$ and $T$ is
obtained from $T'$ by  one Type-2 operation. It is easy to check
that (D1)-(D3) hold for $T$ in this situation.

If $k\ge 3$, we claim that  $T''$ is a generalized star. Suppose to the
contrary that there are two vines, say $S_1$ and $S_2$, whose
lengths are different. Vines $S_2, S_4,\ldots, S_k$ and the path
induced by $S_1\cup S_3\cup \{v\}$ form a minimum path covering of
$T''$. Similarly, vines $S_1, S_4,\ldots,S_k$ and the path induced
by $S_2\cup S_3\cup \{v\}$ form another minimum path covering of
$T''$. Obviously, path sequences of these two minimum path covering
of $T''$ are different as $S_1$ and $S_2$ have different length,
i.e., $T''$ admits different path sequences, and hence $T$ admits
different path sequences, a contradiction. So, all vines in $T''$
have the same number of vertices and $T''$ is a general star. Then
we label vertices $v_1,\ldots,v_k$ by $B$ and label  other non-leaf vertices
in $T''$ by  $A$,  and hence $T$ is obtained from $T'$ by  one
Type-1 operation. It is easy to check that (D1)-(D3) hold for $T$ in
this situation since $e$ is not used in any minimum path covering of
$T$. \vskip 0.2cm

{\it Case 2. $u$ is labeled $B$ in the label tree $T'$.}

Since $u$ is labeled $B$ in the label tree $T'$, $u$ is an end of
the path, induced by the vine (say $S$) containing $u$  in the
 labeled generalized star, in some minimum path covering $\mathcal
{P'}$ of $T'$. Assume that there is a vine in $T''$, say $S_1$,
having different length with $S$.   Let $\mathcal {P''}$ be a
minimum path covering of $T''$ such that  $S_1\cup S_2\cup \{v\}$
forms a path in $\mathcal {P''}$. As $P(T)=P(T')+P(T'')$, $\mathcal
{P'}\cup \mathcal {P''}$ is a minimum path covering of $T$. Now by
swapping $e_1$ and $e$, we get a new path covering of $T$ in which $S\cup S_2\cup \{v\}$ forms a path.
It is
easy to find that it has different path sequence with $\mathcal
{P'}\cup \mathcal {P''}$, a contradiction. Hence,  both $S_i$ and
$S$ have the same number of vertices for all $i\in\{1,2,\cdots, k\}$. We now
label $T''$ as follows: all neighbors of $v$ in $T''$ are labeled by
$B$ and other non-leaf vertices in $T''$ are labeled by $A$.
Therefore, $T''$ is a labeled generalized star and $T$ is obtained from
$T'$ by one Type-3 operation. Using the swapping construction and
Lemmas \ref{lem1}, \ref{lem2}, \ref{lem3}, it is straightforward to
check that (D2)-(D3) hold for $T$.
\end{proof}

Theorem \ref{thm9} and Theorem \ref{thm10} tell us that a tree with
at least three vertices admits unique path sequence if and only if it is
an underlying tree in $\mathcal F$ (An underlying tree in $\mathcal F$ is a tree in $\mathcal F$ deleted its labels). Recall the discussion in the begin
of this section, we immediately obtain  a clear characterization for
trees, whose complements admit multiple island sequences.

\begin{theorem}\label{thm11}
Let $T$ be a tree which is neither a path nor a generalized star.
Then $T^c$ admits multiple  island sequences if and only if $T$ is
not an underlying tree in $\mathcal F$.
\end{theorem}

As we know, if a tree $T$  is neither a star nor a path, then $T^c$
is connected. Hence, Theorem \ref{thm11} settles  Question
\ref{que1} posed by Georges and Mauro \cite{gm2005}. It is obvious
that $2$-sparse trees are not underlying trees in $\mathcal F$.
Therefore, Theorem \ref{thm2} (cf. Theorem 2.8 in \cite{attw2010})
is a  corollary of Theorem \ref{thm11}.

Based on Theorem \ref {thm11} and Lemmas \ref{lem2}, \ref{lem3}, we have
the following algorithm to determine whether its complement of a given tree $T$ has
unique island sequence.

\vspace{4mm}
\textbf{Algorithm DUIS.} Determine whether the complement $T^c$
has unique island  sequence for a given tree $T$. 

\textbf{Input.} A tree $T$.

\textbf{Method.}

\textbf{Step 0:}~\parbox[t]{29em}{Initializes the labels of all
vertices $v\in V(T)$ with $f(v)=0$ and $\ell(v)=O$;}\vspace{8pt}

\textbf{Step 1:}~\parbox[t]{29em}{If~\parbox[t]{28em}{$T$ is a path
$P_k$ with $k \geq 1$, then

If~\parbox[t]{27em}{there is a leaf or isolated vertex $x$ with
$l(x)=A$ or $0<f(x)\neq k$, then

go to \textbf{Step 7};}\vspace{8pt}

else\vspace{8pt}

\hspace{1em} \parbox[t]{24em}{Output ``$T^c$ has unique island
sequence",~and stop.}\vspace{8pt}

endif}\vspace{8pt}

endif}\vspace{8pt}

\textbf{Step 2:}~\parbox[t]{29em}{Let $r$ be a leaf of $T$. $v$ is
the heavy vertex such that $d_T(v,r)$ is the largest and $u$ is the
vertex with $d_T(u,r)=d_T(v,r)-1$. let $e=uv$ and let $T'$~($T''$,
respectively) be the connected component of $T-e$ containing
$u$~($v$, respectively).}\vspace{8pt}

\textbf{Step 3:}~\parbox[t]{26em}{If~\parbox[t]{24em}{$T''$ is
neither a path nor a generalized star, then

go to \textbf{Step 7};}\vspace{8pt}

endif}\vspace{8pt}

\textbf{Step 4:}~\parbox[t]{29em}{If~\parbox[t]{28em}{$T''$ is a
generalized star $S$, let $k$ be the number of vertices of a vine of
$S$.

If~\parbox[t]{27em}{(There is a neighbor $x$ of $v$ in $S$ with
$\ell(x)=A$ or $0<f(x)\neq k$) or (There is a leaf $x$ in $S$ with
$\ell(x)=A$ or $f(x)>0$), then

go to \textbf{Step 7};}\vspace{8pt}

endif\vspace{8pt}

If~\parbox[t]{26em}{$0<f(u)\neq k$, then

$\ell(u):=A$;

$f(u):=0$;}\vspace{8pt}

endif\vspace{8pt}

If~\parbox[t]{26em}{$f(u)=0$ and $\ell(u)=O$, then

$f(u):=k$}\vspace{8pt}

endif\vspace{8pt}

$T:=T'$;

go back to \textbf{Step 1};}\vspace{8pt}

endif}\vspace{8pt}

\textbf{Step 5:}~\parbox[t]{29em}{If~\parbox[t]{28em}{$T''$ is a
path $P_k (k\geq 4)$ such that $v$ is not the middle vertex in
$P_k$, then

If~\parbox[t]{27em}{there is a leaf $x$ in $T''$ with $\ell(x)=A$ or
$0<f(x)\neq k$, then

go to \textbf{Step 7};}\vspace{8pt}

else\vspace{8pt}

\hspace{1em}\parbox[t]{22em}{$\ell(u):=A$;

$f(u):=0$;

$T:=T'$

go back to \textbf{Step 1};}\vspace{8pt}

endif}\vspace{8pt}

endif}\vspace{8pt}

\textbf{Step 6:}~\parbox[t]{29em}{If~\parbox[t]{28em}{$T''$ is a
path $P_k (k\geq 3)$ such that $v$ is the middle vertex in $P_k$,
let $v_1$ and $v_2$ be two neighbors of $v$ in $T''$,

If~\parbox[t]{27em}{there is a leaf $x$ in $T''$ with $l(x)=A$ or
$0<f(x)\neq k$, then

go to \textbf{Step 7};}\vspace{8pt}
endif\vspace{8pt}

If~\parbox[t]{27em}{(There is a vertex $x\in\{u,v_1,v_2\}$ with
$\ell(x)=A$ or $0<f(x)\neq \frac{k-1}{2}$) or (There is a leaf $x$ in
$T''$ with $f(x)=k$),then

$\ell(u):=A$;

$f(u):=0$;}\vspace{8pt}

else\vspace{8pt}

\hspace{1em}\parbox[t]{23em}{$f(u):=\frac{k-1}{2}$;}\vspace{8pt}

endif\vspace{8pt}

$T:=T'$;

go back to \textbf{Step 1};}\vspace{8pt}

endif} \vspace{8pt}

\textbf{Step 7:}~Output ``$T^c$ has multiple  island sequences", and
stop.

\vspace{15pt}
In \textbf{Algorithm DUIS}, two labels on  each vertex $v$, denoted by  $\ell (v)$, $f(v)$, are used.  For convenience,
we first initializes the labels  with   $\ell(v)=O$ and $f(v)=0$ for all vertices $v\in V(T)$.
In Steps 4, 5, 6 of \textbf{Algorithm DUIS}, when the value $A$ is assign to $\ell(u)$, it implies that $u$ must be an internal vertex in  any minimum path covering of $T'$
if $T$ has unique path sequence; When  a positive integer $k$ is assign to $f(u)$, it implies that $P$ must have $k$ vertices
if $T$ has unique path sequence and $u$ is an end of a path $P$ in some minimum path covering of $T'$.  Hence, once $A$ is assign to $\ell(u)$ in Steps 4, 5 and  6, we take $f(u)=0$.
In Step 3, if $T''$ is neither a path nor a generalized star, by Theorem \ref{thm11}, the algorithm outputs ``$T^c$ has multiple island sequences".
In Steps 4, 5 and 6,  the algorithm first checks two labels of each end  of vines of $S$ (if $T''$ is a generalized star $S$ ) or each leaf of path $P_k$
(if $T''$ is a path $P_k$) to determine whether the algorithm output `` $T^c$ has multiple island sequence"; Then determine whether we need to change two labels of $u$.
At last, let $T=T'$ go back to Step 1.
In Step 1, we consider a path $P_k$ with $k\ge 1$. If there is a leaf or isolated vertex $x$ with
$l(x)=A$ or $0<f(x)\neq k$, it is obvious that $T$ has no unique path sequence.   
In fact, since $d_T'(r,v)=d_T(r,v)$ for each vertex $v\in V(T')$, where $r$ is the  selected leaf in Step 2,  Step 2 need to perform only one time.
Note that every vertex of $T$ are used in a constant number in \textbf{Algorithm DUIS}.
Hence, It is easy to see that \textbf{Algorithm DUIS} has complexity $O(|V(T)|)$.

\begin{theorem}
\textbf{Algorithm DUIS} can determine whether its complement $T^c$  of a given tree $T$ has unique island sequence in $O(|V(T)|)$ time.
\end{theorem}

\section{Some invariants of graphs and their complements}

In this section we extend some  results in Section 3 to more general
graphs.

  In \cite{attw2010}, Adams et. al determined the path covering number of
  connected non-cycle $2$-sparse graphs. We now extend their result.

\begin{theorem}\label{thm12}
Let $G$ be a  connected non-cycle graph with $m\ge 1$ edges, $n$
vertices, $h$ heavy edges, and $\ell$ leaves. If  every heavy vertex
in $G$ is adjacent to at least three light vertices, then
$P(G)=\ell+m-h-n$.
\end{theorem}

\begin{proof}
The proof proceeds by induction on $h$, the number of heavy edges in
$G$. If $h=0$, then $G$ is $2$-sparse. By Theorem \ref{thm4} (cf.
Theorem 3.2 in \cite{attw2010}), We know that $P(G)=\ell+m-n$. Let
us assume that $h>1$ and that the result holds for any connected
non-cycle graph with $k$ ($1\le k<h$) heavy edges if its each heavy
vertex has at least three light neighbors. Consider $G$ a
connected non-cycle graph with $m$ edges, $n$ vertices, $h$ heavy
edges, and $\ell$ leaves. Assume that each heavy vertex in $G$ has at least three light neighbors.
 Let $e=uv$ be a heavy
edge in $G$ and let $G_1,\ldots,G_t$ be connected components of
$G-e$, where $1\le t\le 2$. By Lemma \ref{lem4},
$P(G)=P(G-e)=\sum\limits_{i=1}^{t} P(G_i)$. Since $u$ ($v$,
respectively) is adjacent to at least three light vertices, each
$G_i$ is not a cycle and  $G_i$ satisfies the induction hypothesis,
and hence $P(G_i)=\ell_i+m_i-h_i-n_i$, where $m_i$ ($n_i$, $h_i$,
$\ell_i$, respectively) is the number of edges (vertices, heavy
vertices, leaves, respectively) in $G_i$. Note that the following
equalities hold:

$$\sum\limits_{i=1}^{t}\ell_i=\ell \ \ \ \ \mbox{and} \ \ \ \sum\limits_{i=1}^{t}m_i=m-1;$$
$$\sum\limits_{i=1}^{t}n_i=n \ \ \ \ \mbox{and}\ \ \ \ \sum\limits_{i=1}^{t}h_i=h-1.$$

Therefore, $P(G)=\ell+m-h-n$ and the result follows.
\end{proof}

  Similarly, the following result is a direct corollary of Theorem
 \ref{thm12}, Theorem \ref{thm1} and Theorem \ref{thm}.

\begin{corollary}
Let $G$ be a  connected non-cycle graph with $m\ge 1$ edges, $n$
vertices, $h$ heavy edges, and $\ell$ leaves. If  every heavy vertex
in $G$ is adjacent to at least three light vertices and $\ell+m\ge
h+n+2$, then $\lambda(G^c)=\ell+m-h-2$ and $\rho(G^c)=\ell+m-h-n-1$.
\end{corollary}

A {\it complete graph } $K_n$ is a graph of order $n\ge 2$ in which
every two vertices are adjacent. A vertex $v\in V(G)$ is a {\it
cut-vertex} if deleting $v$ and all edges incident to it increases
the number of connected components. A {\sl block} of $G$ is a
maximal connected subgraph of $G$ without cut-vertex. A {\it block
graph} is a connected graph whose blocks are complete graphs. If
every block is $K_2$, then it is a tree. 

Given a nontrivial tree $T$, $\mathcal {G}(T)$ is defined as a
family of block graphs obtained from $T$ by expanding each edge of
$T$ into a complete graph of arbitrary order. It is obvious that $T\in \mathcal {G}(T)$.
Figure \ref{fig5} illustrates a graph $G$ obtained from $P_3$ by expanding
edges $e$, $f$ into $K_3$, $K_4$, respectively.

\begin{figure}[!htbp]
\centering
\begin{tikzpicture}[
  vertex/.style={inner sep=1pt,circle,draw}]
\path node[vertex,label=90:$u$](1) at(90:1.3) {}
      node[vertex,label=270:$w$](2) at(270:1.3) {}
      node[vertex,label=180:$v$](3) at(0:0) {};
\path[xshift=5cm]
      node[vertex,label=90:$u$](5) at(90:1.3) {}
      node[vertex,label=270:$w$](7) at(270:1.84) {}
      node[vertex,label=0:$v$](4) at(0:0) {}
      node[vertex](6) at(-45:1.3) {}
      node[vertex](8) at(225:1.3) {}
      node[vertex](9) at(135:1.3) {};
\draw (1) -- (3) -- (2)
      (4) -- (6) -- (7) -- (8) -- (4) -- (5) -- (9) -- (4) -- (7)
      (6) -- (8);
\node at($(3)!.5!(4)$) {$\Longrightarrow$};
\end{tikzpicture}
\caption{}\label{fig5}
\end{figure}

Now we give a result to extend the result in Theorem \ref{thm3} (cf.
Theorem 2.4 in \cite{attw2010}).

\begin{theorem}\label{thm13}

Let $T$ be a $2$-sparse tree with $\ell\ge 2$ leaves. Then
$P(G)=\ell-1$ for any graph $G\in \mathcal{ G }(T)$.
\end{theorem}

\begin{proof}
Let $G\in \mathcal{ G }(T)$. As $G$ is an expansion of $T$, it is
easy to see that $P(G)\ge P(T)=\ell-1$. Consider an arbitrary
minimum path covering of $T$, we will use this minimum path covering
to construct a path covering of $G$ with exactly $\ell-1$ paths,
which will imply that $P(G)=\ell-1$. Let $e=uv$ be an arbitrary
edge of $T$. Assume that the vertices of the block of $G$ replacing
$e$ are $u,x_1,\ldots,x_t, v$ (For convenience,  two cut-vertices in
this block are still called as $u,v$). We now construct a path covering of
$G$ as follows: If $e$ is contained in a path $P=\cdots uv\cdots$ in
this minimum path covering of $T$, then we construct a path
$P'=\cdots ux_1\ldots x_tv\cdots$; If $e$ is not contained in any
path in this minimum path covering of $T$, by Lemma \ref{lem1}, one
of $u,v$ is a heavy vertex. Without loss of generality, we assume
that $v$ is a heavy vertex. As $T$ is $2$-sparse, we know that $u$
is a light vertex in $T$ and hence $u$ is one end of a path, say $P$, in this
minimum path covering of $T$. Then we construct a path $P'=Px_1\cdots x_t$.
\end{proof}

Our final corollary below determines the $\lambda$ and $\rho$ of
complements of graphs satisfying the condition of Theorem
\ref{thm13}.

\begin{corollary}
Let $T$ be a $2$-sparse tree with $\ell\ge 3$ leaves. Then
$\lambda(G^c)=n+\ell-3$ and $\rho(G^c)=\ell-2$ for any $G\in
\mathcal{ G }(T)$, where $n$ is the order of $G$.
\end{corollary}

At the end of this section, we shall point out that it is
interesting to establish the path covering numbers of general
$2$-sparse graphs.

\section{Conclusions}

In this paper, we determined the path covering number of  some
families of trees and  tree-like graphs. Additionally, we determined
$\lambda$ and $\rho$ for its complements. We  also established a
constructive characterization  for trees whose complements admit
unique island sequences. A linear-time algorithm was also given to determine
whether the complement of a given tree $T$ admits unique island sequence. Hence,
 an open question in \cite{gm2005}
was answered. Our work generalized most of results in
\cite{attw2010}. We hope these results will be extended to include
more general trees and graphs.

\end{document}